%% file: totaling.tex
\documentclass{amsart}
\usepackage{amsmath,amsthm,amssymb,amsfonts}
\usepackage{array}
\usepackage{tabularx}
\usepackage{graphicx}
\usepackage{mathrsfs}
\usepackage{cite}
\usepackage{mathtools}
\usepackage{hyperref}
\usepackage[abbrev,msc-links]{amsrefs}
\usepackage[all,cmtip]{xy}

\input{mymacros.tex}

\begin{document}
\title[On the Image of the Totaling Functor]{On the Image of the Totaling Functor}
\author[K.\ A.\ Beck]{Kristen A.\ Beck}
\address{Kristen A.\ Beck, Department of Mathematics, University of Arizona, 617 N.\ Santa Rita Ave., Tucson, AZ, 85719, U.S.A.}
\email{kbeck@math.arizona.edu}

\thanks{{\em Date.} \today.}
\thanks{This research was partially supported by NSA Grant H98230-07-1-0197 and NSF GK-12 Program Grant 0841400.}

\begin{abstract}
Let $A$ be a DG algebra with a trivial differential over a commutative unital ring. This paper investigates the image of the totaling functor, defined from the category of complexes of graded $A$-modules to the category of DG $A$-modules. Specifically, we exhibit a special class of semifree DG $A$-modules which can always be expressed as the totaling of some complex of graded free $A$-modules. As a corollary, we also provide results concerning the image of the totaling functor when $A$ is a polynomial ring over a field.
\end{abstract}

\maketitle

\section*{Introduction}

Let $A$ be a DG algebra over a commutative unital ring. Furthermore, let $\DG(A)$ denote the category of DG $A$-modules and their degree zero chain maps, and let $\Ch\Gr(A^\natural)$ denote the category of (co)chain complexes of graded $A^\natural$-modules and their degree zero chain maps. The motivation for the work in this paper is to better understand the difference between these categories.  In the most general case, one can easily see that $\DG(A)$ is a much `richer' category than $\Ch\Gr(A^\natural)$ --- indeed, its objects take into account two differentials rather than just one.  However, when $A$ has a trivial differential (so that $A=A^\natural$), the difference between $\DG(A)$ and $\Ch\Gr(A^\natural)$ is not so striking.

The goal of this paper is to address the latter scenario by studying the image of the so-called \emph{totaling functor} $\Tot\colon \Ch\Gr(A^\natural)\to\DG(A)$ in the case that $A$ has a trivial differential. Given a cochain complex $X$ of graded modules over such a DG algebra $A$, one defines $\Tot X$ to be the complex whose underlying graded $A^\natural$-module structure is given by
\[
\left(\Tot X\right)^\natural:=\bigoplus_{i\in\ZZ}\shift^{-i}X^i
\]
and whose differential follows in a natural way from that of $X$.  Upon defining an $A$-action on this complex, $\Tot X$ admits the structure of a DG $A$-module. The primary question we consider is whether the totaling functor is surjective.

In Theorem \ref{thm:crossing}, we provide a necessary and sufficient condition for a DG module over a DG algebra $A$ with a trivial differential to be equal to the totaling of some complex of graded free $A^\natural$-modules.  The constructive nature of the proof of this result furthermore allows us to express a `totaling pre-image' for DG modules which lie in the image of the totaling functor.  To answer the question of whether the totaling functor is surjective, we restrict to the case that $A$ is a polynomial ring over a field.  In Example \ref{ex:notinim}, we exhibit a DG module over a polynomial ring in two (or more) variables which does not satisfy the condition specified by Theorem \ref{thm:crossing}, and therefore does not lie in the image of the functor.  Moreover, in Theorem \ref{thm:1var}, we illustrate that every DG module over a polynomial ring in one variable is quasiisomorphic to the totaling of some complex of graded $A^\natural$-modules.

\section{Background}

\numberwithin{theorem}{subsection}
\numberwithin{equation}{theorem}

The results in this paper will assume a working understanding of differential graded (DG) algebras and their modules.  For a thorough treatment of this subject, the reader is referred to \cite{Av10}, \cite{AvFoxHal}, or \cite{Ke06}.  In what follows, $A$ is assumed to be a DG algebra over a commutative unital ring.

\subsection{Semifree DG modules}

We begin by introducing a class of DG modules which generalize free modules over a ring.  They will form the basic structures necessary for the construction of (counter-)examples in the sequel, and will also be essential to the statement of our main result.

\begin{definition}
Let $M$ be a DG $A$-module. A subset $E\subseteq M^\natural$ is called a \emph{semibasis} for $M$ if
\begin{enumerate}
\item $E$ is a basis for $M^\natural$ over $A^\natural$, and
\item $E=\bigsqcup_{d\in\NN}E_d$ (a disjoint union) such that
\[
\dell(E_d)\subseteq A\left(\bigsqcup_{i<d}E_i\right)
\]
for all $d\in\NN$.
\end{enumerate}
A DG module that possesses a semibasis is said to be \emph{semifree}.
\end{definition}

\begin{proposition}\label{lem:semifree}\cite{AvFoxHal}*{8.2.3}
Let $M$ be a DG $A$-module.  The following are equivalent.
\begin{enumerate}
\item $M$ is semifree.
\item $M^\natural$ has a well-ordered basis $E$ over $A^\natural$ such that for each $e\in E$
\[
\dell(e)\in A\left(\{e'\in E\var e'<e\}\right).
\]
\end{enumerate}
\end{proposition}

\begin{proof}
To show $(1)\Rightarrow(2)$, let $E=\bigsqcup_{d\in\NN}E_d$ be a semibasis for $M$ over $A$.  For each $d\in\NN$, impose an ordering on $E_d$, and further suppose that whenever $d'<d$, $e'<e$ for every $e'\in E_{d'}$ and $e\in E_d$.  This implies that $E$ has an ordering with the desired property.

On the other hand, in order to show that $(2)\Rightarrow(1)$, suppose that $E$ is a well-ordered basis for $M^\natural$ over $A^\natural$ such that $\dell(e)\in A\left(\{e'\in E\var
e'<e\}\right)$ for every $e\in E$.  Set $E_{-1}=\varnothing$ and $M_{-1}=\{0\}$, and for each $d\in\NN$, recursively define $E_d$ and $M_d$ in the following manner.
\begin{align*}
E_d&:=\left\{e\in E\setminus\textstyle{\bigcup_{i<d}E_i}\var\dell(e)\in M_{d-1}\right\}\\
M_d&:=A\left(\bigcup_{i\leq d}E_i\right)
\end{align*}
By this construction, the $E_d$ are mutually disjoint.  Furthermore, the well-ordering of $E$ implies that $E=\bigsqcup_{d\in\NN}E_d$, and the result follows.
\end{proof}

\begin{remark}
A \emph{semifree resolution} of a DG $A$-module $M$ is a quasiisomorphism $\pi\colon F\to M$ of DG $A$-modules where $F$ is semifree.  In \cite{AvFoxHal}, Avramov, Foxby, and Halperin show that every DG module possesses a semifree resolution.  This fact makes possible the study of differential graded homological algebra.
\end{remark}

\subsection{The totaling functor}

Throughout this section, suppose that $A$ has a trivial differential.\footnote{Indeed, $\Tot$ can be defined on the category $\Ch\DG(A)$ over an arbitrary DG algebra $A$ (see \cite{AvFoxHal}*{Section 7.2} for details).  However, for the purposes of our work, it suffices to define the functor in the present setting.}
Let $X$ be a cochain complex in $\Ch\Gr(A^\natural)=\Ch\Gr(A)$, and denote the internal grading of $X^i$ with a subscript; that is, $X^i=\bigoplus_{j\in\ZZ} X_j^i$.  Now define the \emph{totaling} of $X$ to be the complex $\Tot X=\left((\Tot X)^\natural,\dell^{\Tot X}\right)$, whose underlying graded structure is given by
\[
\left(\Tot X\right)^\natural:=\bigoplus_{i\in\ZZ}\shift^{-i} X^i.
\]
Therefore, the $d$th homological component of $\Tot X$ is given by
\[
(\Tot X)_d=\bigoplus_{i\in\ZZ}X^i_{d+i}
\]
and one defines the $A$-module structure on $\Tot X$ according to
\[
a(\susp^{-i}x^i)_{i\in\ZZ}:=\left(\susp^{-i}((-1)^{|a|i}ax^i)\right)_{i\in\ZZ}
\]
for each $a\in A$ and $(\susp^{-i}x^i)_{i\in\ZZ}\in\Tot X$.  Moreover, the differential on $\Tot X$ is defined by
\[
\dell^{\Tot X}\!\left((\susp^{-i} x^i)_{i\in\ZZ}\right):=\left(\susp^{-i-1}\dell_X^{i}(x^{i})\right)_{i\in\ZZ}
\]
for each $(\susp^{-i}x^i)_{i\in\ZZ}\in\Tot X$.

It is now straightforward to check that $\Tot X$ is in fact a DG module over $A$.  Indeed, for any $a\in A$ and $(\susp^{-i} x^i)_{i\in\ZZ}\in\Tot X$, one has
\begin{align*}
\dell^{\Tot X}\!\left(a(\susp^{-i} x^i)_{i\in\ZZ}\right)&=\dell^{\Tot X}\!\left(\left(\susp^{-i}((-1)^{|a|i}ax^i)\right)_{i\in\ZZ}\right)\\
&=\left(\susp^{-i-1}\dell_X^{i}((-1)^{|a|i}ax^{i})\right)_{i\in\ZZ}\\
&=\left((-1)^{|a|i}\susp^{-i-1}(a\,\dell_X^{i}(x^{i}))\right)_{i\in\ZZ}\\
&=\left((-1)^{|a|i}(-1)^{|a|(i+1)}a\,\susp^{-i-1}\dell_X^i(x^i)\right)_{i\in\ZZ}\\
&=(-1)^{|a|}a\left(\susp^{-i-1}\dell_X^{i}(x^{i})\right)_{i\in\ZZ}\\
&=(-1)^{|a|}a\,\dell^{\Tot X}\left((\susp^{-i}x^i)_{i\in\ZZ}\right)
\end{align*}
so that the Leibniz rule is satisfied.

For any morphism $\mu\colon X\to Y$ of complexes of graded $A$-modules, define the following map.
\begin{align*}
\Tot\mu\colon \Tot X &\to\Tot Y\\
(\susp^{-i}x^i)_{i\in\ZZ}&\mapsto(\susp^{-i}\mu^i(x^i))_{i\in\ZZ}
\end{align*}
To see that $\Tot\mu$ is a morphism of DG $A$-modules, note that for each $x^i\in X^i_j$, $\mu^i(x^i)\in Y^i_j$. It follows that totaling is functorial.  Thus, $\Tot\colon \Ch\Gr(A)\to\DG(A)$ is called the \emph{totaling functor}.

\begin{remark}
To define the totaling of a \emph{chain} complex in $\Ch\Gr(A)$, one should note that $\bigoplus_{i\in\ZZ}\shift^{-i}X^i\cong\bigoplus_{i\in\ZZ}\shift^i X_i$.
\end{remark}

\section{Results}

\numberwithin{theorem}{section}
\numberwithin{equation}{theorem}

In this section, we turn our attention to the image of the totaling functor.  Unless otherwise stated, $A$ shall denote a DG algebra with a trivial differential and $k$ shall denote a field.  Furthermore, one may assume that all polynomial rings have the standard grading.

\begin{definition}
Let $A$ be an arbitrary DG algebra, and $M$ a semifree DG module with semibasis $E$ over $A$. Define a family of disjoint sets recursively by
\begin{align}\label{eq:crossing}
\begin{split}
&\mathcal{E}_0:=\{e\in E\mid\dell(e)=0\}\\
&\mathcal{E}_\ell:=\left\{e\in E\mid 0\neq\dell(e)\in A\mathcal{E}_{\ell-1}\right\}
\end{split}
\end{align}
for all $\ell\in\ZZ^+$.  Notice that since $M$ is semifree, $\mathcal{E}_0$ is nonempty.  Now consider the following containment of sets.
\begin{equation}\label{eq:crossingineq}
\bigsqcup_{\ell\in\NN}\mathcal{E}_\ell\subseteq E
\end{equation}
If the containment in \eqref{eq:crossingineq} is strict, we say that $E$ has \emph{crossing}.
\end{definition}

\begin{remark}\label{rmk:crossmatrix}
One can determine whether a particular semibasis for a finitely-generated DG module has crossing by simply examining the matrix representing $\dell$.  Indeed, if $D$ is the matrix representing $\dell$ with respect to the semibasis $E$, then $D$ must \emph{a priori} be strictly upper-triangular.  This follows directly from the definition of semibasis.  If one further supposes that $E=\bigsqcup_{\ell=0}^m\mathcal{E}_\ell$ has no crossing, then $D$ will take a block super-diagonal form.  Specifically,
\[
D=\begin{bmatrix}
Z_0&D_0\\
&Z_1&D_1\\
&&Z_2&\\
&&&^{\ddots}&D_{m-1}\\
&&&&Z_m
\end{bmatrix}
\]
where $Z_i$ is a $|\mathcal{E}_i|\times |\mathcal{E}_i|$ zero matrix, $D_i$ is a $|\mathcal{E}_i|\times |\mathcal{E}_{i+1}|$ matrix with entries in $A^\natural$, and the non-specified entries are zeros.
\end{remark}

To illustrate the concept of crossing, we provide an example which exhibits semibases with and without crossing for the same semifree DG module.

\begin{example}\label{ex:cross/nocross}
Let $A=k[x,y,z]$ and consider the rank four semifree DG $A$-module $M$ with semibasis $E=\{e_1,e_2,e_3,e_4\}$ such that $|e_1|=0$, $|e_2|=2$, $|e_3|=3$, $|e_4|=5$, and
\begin{align*}
&\dell(e_1)=0\\
&\dell(e_2)=xe_1\\
&\dell(e_3)=yze_1\\
&\dell(e_4)=xz^3e_1+yze_2-xe_3.
\end{align*}
Then $\mathcal{E}_0=\{e_1\}$, $\mathcal{E}_1=\{e_2,e_3\}$, and $\mathcal{E}_\ell=\varnothing$ for $\ell\geq 2$. Thus, the strict containment
\[
\bigsqcup_{\ell\in\NN}\mathcal{E}_\ell\subsetneq E
\]
implies that $E$ has crossing.

On the other hand, define a semibasis $E'$ for $M$ by $e'_i=e_i$ for $1\leq i\leq 3$ and $e'_4=e_4-z^3e_2$.  The action of $\dell^M$ on $E'$ is given by
 \begin{align*}
&\dell(e'_1)=0\\
&\dell(e'_2)=xe'_1\\
&\dell(e'_3)=yze'_1\\
&\dell(e'_4)=yze'_2-xe'_3.
\end{align*}
In this case, we obtain $\mathcal{E}_0=\{e'_1\}$, $\mathcal{E}_1=\{e'_2,e'_3\}$, $\mathcal{E}_2=\{e'_4\}$, and $\mathcal{E}_\ell=\varnothing$ for $\ell\geq 3$. Therefore, we have an equality of sets
\[
\bigsqcup_{\ell\in\NN}\mathcal{E}_\ell=E'
\]
which implies that $E$ has no crossing.
\end{example}

Next we state our main result, which gives a necessary and sufficient condition for a DG $A$-module to be equal to the totaling of some complex of graded free $A$-modules.

\begin{theorem}\label{thm:crossing}
Let $M$ be a semifree DG module over a DG algebra $A$ with trivial differential. Then $M$ has a semibasis without crossing if and only if there exists a bounded-below complex $X$ of graded free $A$-modules such that $\Tot X=M$.
\end{theorem}

\begin{proof}
Let $E=\bigsqcup_{d\in\NN}E_d$ be a semibasis of $M$ over $A$.  First suppose that $E$ has no crossing, implying that $E=\bigsqcup_{\ell\in\NN}\mathcal{E}_\ell$, where
\begin{align*}
&\mathcal{E}_0=\{e\in E\var\dell^M(e)=0\}\\
&\mathcal{E}_\ell=\left\{e\in E\var0\neq\dell^M(e)\in A\mathcal{E}_{\ell-1}\right\}
\end{align*}
for all $\ell\in\ZZ^+$.  Now define a (possibly infinite) sequence $X$ of homomorphisms of graded $A$-modules by
\[
X\colon \qquad\cdots\to\shift^{-2}A\mathcal{E}_2\xra{\dell^X_2}\shift^{-1}A\mathcal{E}_1\xra{\dell^X_1}A\mathcal{E}_0\to 0
\]
where, for each $\ell\in\NN$ and $e\in\mathcal{E}_\ell$, one has that
\[
\dell^{X}_\ell\big((0,\ldots,0,\susp^{-\ell}e,0,\ldots,0)\big)=\susp^{-\ell+1}\dell^M(e)\subseteq\shift^{-\ell+1}A\mathcal{E}_{\ell-1}.
\]
By construction, $\dell^{X}_{\ell+1}\circ\dell^{X}_\ell=0$ for all $\ell\in\NN$, implying that $X\in\Ch\Gr(A)$. To see that $\Tot X=M$, notice that
\begin{align*}
(\Tot X)^\natural&=\bigoplus_{\ell\in\NN}\shift^{\ell}X_\ell\\
&=\bigoplus_{\ell\in\NN}\!\shift^{\ell}\shift^{-\ell}A\mathcal{E}_\ell\\
&\cong\bigoplus_{\ell\in\NN}A\mathcal{E}_\ell\\
&=M^\natural
\end{align*}
and that
\begin{align*}
\dell^{\Tot X}(\susp^{\ell}\susp^{-\ell}e)&=\susp^{\ell-1}\dell^X_\ell(\susp^{-\ell}e)\\
&=\susp^{\ell-1}\left(\susp^{-\ell+1}\dell^M(e)\right)\\
&=\dell^M(e)
\end{align*}
for all $\ell\in\NN$ and $e\in\mathcal{E}_\ell$.  The result follows.

On the other hand, suppose that $X$ is a bounded-below complex of graded free $A$-modules such that $\Tot X=M$, and let $\widetilde{E}_i$ be a basis for $X_i$ over $A$. That is,
\[
X\colon \quad \cdots\to A\widetilde{E}_{2}\xra{\dell_{2}^X} A\widetilde{E}_1\xra{\dell_1^X} A\widetilde{E}_{0}\to 0
\]
where
\[
M^\natural=\left(\Tot X\right)^\natural=\bigoplus_{i\in\NN}\shift^iA\widetilde{E}_i\quad\text{ and }\quad\dell^M\left((\sigma^i e)_{i\in\NN}\right)=\left(\sigma^{i-1}\dell^X_i(e)\right)_{i\in\NN}
\]
for each $e\in \widetilde{E}_i$. Now define
\begin{align*}
\widetilde{\mathcal{E}}_0&:=\bigsqcup_{i\in\ZZ}\left\{e\in\widetilde{E}_i\mid\dell_i^X(e)=0\right\}\\
\widetilde{\mathcal{E}}_\ell&:=\bigsqcup_{i\in\ZZ}\left\{e\in\widetilde{E}_i\mid 0\neq\dell_i^X(e)\in A\widetilde{\mathcal{E}}_{\ell-1}\right\}
\end{align*}
for each $\ell\in\ZZ^+$. Note that one can write $\widetilde{\mathcal{E}}_\ell=\bigsqcup_{i\in\NN}\widetilde{\mathcal{E}}_{i,\ell}$
where $\widetilde{\mathcal{E}}_{i,\ell}\subseteq\widetilde{E}_i$ for each $\ell\in\NN$. Let $\mathcal{E}_\ell:=\bigoplus_{i\in\NN}\shift^{i}\widetilde{\mathcal{E}}_{i,\ell}$. We will show that $E:=\bigsqcup_{\ell\in\NN}\mathcal{E}_\ell$ is a semibasis for $M$ which does not have crossing.

To this end, we first note that $E$ is a basis for $M^\natural$ over $A$ by construction.  Furthermore, if $(\sigma^i e)_{i\in\NN}\in\mathcal{E}_\ell$, then
\[
\dell^M\left((\sigma^ie)_{i\in\NN}\right)=\left(\sigma^{i-1}\dell^X_i(e)\right)_{i\in\NN}\in\bigoplus_{i\in\NN}\shift^{i-1}A\widetilde{\mathcal{E}}_{i-1,\ell-1}= A\mathcal{E}_{\ell-1}
\]
since $e\in\widetilde{\mathcal{E}}_{i,\ell}$. This implies that $E$ is a semibasis for $M$ without crossing.

\end{proof}

To see that the statement of Theorem \ref{thm:crossing} is not trivial, consider the following example.

\begin{example}\label{ex:notinim}
Let $A=k[x_1,\ldots,x_d]$ where $d\geq 2$, and consider the rank four semifree DG $A$-module $M$ given by
$M^\natural=Ae_1\oplus Ae_2\oplus Ae_3\oplus Ae_4$ such that $|e_1|=0$, $|e_2|=3$, $|e_3|=4$, $|e_4|=8$, and
\begin{align*}
&\dell(e_1)=0\\
&\dell(e_2)=x_1x_2e_1\\
&\dell(e_3)=x_2^3e_1\\
&\dell(e_4)=x_1^7e_1-x_2^4e_2+x_1x_2^2e_3.
\end{align*}
Note that $E=\{e_1,e_2,e_3,e_4\}$ has crossing. This can be verified by examining the matrix representation
\[
D=\begin{bmatrix}
0&x_1x_2&x_2^3&x_1^7\\
0&0&0&-x_2^4\\
0&0&0&x_1x_2^2\\
0&0&0&0
\end{bmatrix}
\]
of $\dell$ with respect to $E$. We will show that there does not exist a semibasis for $M$ without crossing.

To this end, suppose that $E'$ is another semibasis for $M$, and denote by $P$ the change-of-basis matrix; that is, $P\colon E\text{-coordinates}\to E'\text{-coordinates}$. If $D'$ is the matrix representing $\dell$ with respect to $E'$, then $D'=PD$ must take the form prescribed by Remark \ref{rmk:crossmatrix}.  Letting $P=\begin{bmatrix}p_{ij}\end{bmatrix}$, one can write $D'$ as follows.
\[
D'=PD=
\begin{bmatrix}
0&p_{11}x_1x_2&p_{11}x_2^3&p_{11}x_1^7-p_{12}x_2^4+p_{13}x_1x_2^2\\
0&p_{21}x_1x_2&p_{21}x_2^3&p_{21}x_1^7-p_{22}x_2^4+p_{23}x_1x_2^2\\
0&p_{31}x_1x_2&p_{31}x_2^3&p_{31}x_1^7-p_{32}x_2^4+p_{33}x_1x_2^2\\
0&p_{41}x_1x_2&p_{41}x_2^3&p_{41}x_1^7-p_{42}x_2^4+p_{43}x_1x_2^2
\end{bmatrix}
\]
Now, since $E'$ is \emph{a priori} a semibasis for $M$, it should be true that $D'$ is strictly upper-triangular.  This implies that $p_{21}=p_{31}=p_{41}=0$ and $p_{41}x_1^7-p_{42}x_2^4+p_{43}x_1x_2^2=-p_{42}x_2^4+p_{43}x_1x_2^2=0$. Furthermore, since $P$ needs to be invertible, $p_{11}$ must be a unit.  With these relations given, $D'$ now takes the form
\[
D'=
\begin{bmatrix}
0&\alpha x_1x_2&\alpha x_2^3&\alpha x_1^7-p_{12}x_2^4+p_{13}x_1x_2^2\\
0&0&0&p_{21}x_1^7-p_{22}x_2^4+p_{23}x_1x_2^2\\
0&0&0&p_{31}x_1^7-p_{32}x_2^4+p_{33}x_1x_2^2\\
0&0&0&0
\end{bmatrix}
\]
for some $\alpha\in k$. Supposing that $E'$ has no crossing, it follows by Remark \ref{rmk:crossmatrix} that $\alpha x_1^7-p_{12}x_2^4+p_{13}x_1x_2^2=0$.  But this implies that $px_2^4-qx_1x_2^2=x_1^7$ for some $p,q\in k[x_1,\ldots,x_d]$, which cannot be true. Therefore, $E'$ must have crossing.

\end{example}

\begin{remark}
A result of Avramov and Jorgensen \cite{AvJo} can furthermore be used to show that the semifree DG module exhibited in Example \ref{ex:notinim} is, more generally, not in the image of the totaling functor $\Tot\colon \Der\Gr(A)\to\Der\DG(A)$ defined on the respective derived categories. The reader is referred to \cite{Wei94}*{Chapter 10} for a constructive definition of derived categories.
\end{remark}

\begin{corollary}\label
Let $A$ be a polynomial ring in more than one variable over a field. Then the functor $\Tot\colon\Ch\Gr(A)\to\DG(A)$ is not surjective.
\end{corollary}

The constructive nature of the proof of Theorem \ref{thm:crossing} makes it possible for one to cook up a `totaling pre-image' for any semifree DG module which admits a basis without crossing.  The following example demonstrates this fact.

\begin{example}
Let $M$ be the rank four semifree DG module with semibasis $E$ over $A=k[x,y,z]$ defined in Example \ref{ex:cross/nocross}.  As $E$ has no crossing, the construction in the proof to Theorem \ref{thm:crossing} yields a complex
\[
X\colon \qquad 0\to\shift^{-2}Ae_4\xra{
\left[\footnotesize{\begin{array}{@{}r@{\hspace{0.25em}}} yz\\[0.02in]-x
\end{array}}\right]
}\shift^{-1}(Ae_2\oplus Ae_3)\xra{ \left[\footnotesize{\begin{array}{@{\hspace{0.25em}}c@{\hspace{1em}}c@{\hspace{0.25em}}} x&yz
\end{array}}\right]
}Ae_1\to 0
\]
of graded $A$-modules such that $\Tot X=M$.
\end{example}

As one might have guessed, it is not always possible to cook up a semifree DG module whose semibases are guaranteed to have crossing.  For example, if a DG $A$-module has rank no more than three over $A$, it is impossible to define a differential which gives the the semibasis crossing.  The following corollary makes this precise.

\begin{corollary}\label{cor:rankleq3}
Let $A$ be a graded domain, and suppose that $M$ is a semifree DG $A$-module.  If $\rank_A M\leq 3$ then there exists a complex $X$ of graded $A$-modules such that $\Tot X=M$.
\end{corollary}

\begin{proof}
Let $E=\{e_1,\ldots,e_d\}$ be a well-ordered basis for $M^\natural$ over $A$.  Since $M$ is semifree, we obtain the following possible (non-trivial) forms for its
differential, where $a_{ij}\in A$ for each $i,j$.
\begin{center}
\begin{tabular}{lllllll}
$d=1$&$\Rightarrow$&$\dell(e_1)=0$\\[.2cm]
$d=2$&$\Rightarrow$&$\dell(e_1)=0$\\[0.1cm]
&&$\dell(e_2)= a_{12}e_1$\\[.2cm]
$d=3$&$\Rightarrow$&$\dell(e_1)=0$&&&&$\dell(e_1)=0$\\[0.1cm]
&&$\dell(e_2)= a_{12}e_1$&& or &&$\dell(e_2)=0$\\[0.1cm]
&&$\dell(e_3)= a_{13}e_1$&&&&$\dell(e_3)= a_{13}e_1+ a_{23}e_2$
\end{tabular}
\end{center}
Note that in each case, $E$ has no crossing. The result follows by Theorem \ref{thm:crossing}.
\end{proof}

We now turn our attention to the image of the totaling functor in the case that $A$ is a polynomial ring in one variable.  In this setting, it turns out that using Theorem \ref{thm:crossing} to find semifree DG modules which are not in the image of the totaling functor is not so straightforward.  Accordingly, we take a slightly different approach here.


The next lemma, which will follows directly from the structure theorem for finitely generated modules over a principal ideal domain.

\begin{lemma}\label{lem:resh}
Let $M$ be a DG module which is $n$-generated over $A=k[x]$.  Then there exist integers $0\leq s\leq d$ and $1\leq t\leq d$ such that $\HH(M)$ has a graded minimal free resolution over $A$ given by
\[
0\to\bigoplus_{j=1}^s\shift^{c_j}A\xra{
\left[\begin{array}{ccc}
h_1&&\text{\huge{$0$}}\\
&\ddots\\
\text{\huge{$0$}}&&h_s\\[.5pc]
\hline\\
&\text{\huge{$0$}}
\end{array}\right]}
\bigoplus_{i=1}^t\shift^{r_i}A\to\HH(M)\to 0
\]
for some integers $r_i,c_j$ for $1\leq i\leq t$ and $1\leq j\leq s$, and where each $h_i=x^{c_i-r_i}\in Ax$ for each $1\leq i\leq s$.
\end{lemma}

\begin{theorem}\label{thm:1var}
Let $A$ be a polynomial ring in one variable over a field.  Then every semifree DG $A$-module is quasiisomorphic to the totaling of some complex of graded free $A$-modules.  Specifically, the functor $\Tot\colon\Der\Gr(A)\to\Der\DG(A)$ is surjective.
\end{theorem}

\begin{proof}
Let $M$ be a rank $d$ semifree DG module over $A=k[x]$.  By Lemma \ref{lem:resh}, one can assume the homology of $M$ to have the form
\begin{equation}\label{eq:homiso}
\HH(M)\cong\bigoplus_{i=1}^s\frac{\shift^{r_i}A}{h_i\shift^{c_i}A}
\oplus\bigoplus_{i=s+1}^t\shift^{r_i}A
\end{equation}
for some $1\leq t\leq d$ and $0\leq s\leq t$, and positive integers $r_i,c_j$.  (Of course, $s=0$ corresponds to the case that $\HH(M)$ is a free $A$-module.) Consider the deleted minimal graded free resolution $F$ of $\HH(M)$ given as follows.
\[
\qquad\qquad\qquad F\colon \hspace{-1in}
\begin{split}
\xymatrixcolsep{2pc}\xymatrixrowsep{0pc}\xymatrix{
0\ar@{->}[r]&\shift^{c_1}A\ar@{->}[r]^{h_1}&\shift^{r_1}A\ar@{->}[r]&0\\
&\oplus&\oplus\\
&\vdots&\vdots\\
&\oplus&\oplus\\
0\ar@{->}[r]&\shift^{c_s}A\ar@{->}[r]^{h_s}&\shift^{r_s}A\ar@{->}[r]&0\\
&\oplus&\oplus\\
&0\ar@{->}[r]&\shift^{r_{s+1}}A\ar@{->}[r]&0\\
&\oplus&\oplus\\
&\vdots&\vdots\\
&\oplus&\oplus\\
&0\ar@{->}[r]&\shift^{r_t}A\ar@{->}[r]&0\\}
\end{split}
\]

To complete the proof, we will show that $\Tot F\simeq M$.  While it is clear that $\HH(\Tot F)\cong\HH(M)$, it remains to exhibit a chain map (or sequence thereof) which induces this isomorphism.  To this end, for each $1\leq i\leq t$, let $G_i$ be the subcomplex which is given by the $i$th summand of $F$.
\[
G_i\colon \qquad\left\{
\begin{array}{ll}
0\to\shift^{c_i}A\xra{h_i}\shift^{r_i}A\to 0 &\qquad \text{if }1\leq i\leq s\\[0.1cm]
0\to\shift^{r_i}A\to 0 &\qquad \text{if }s<i\leq t
\end{array}\right.
\]
Our goal is to define a family of chain maps $\mu^i\colon \Tot G_i\to M$ such that the chain map $\mu\colon \Tot F\to M$ given by
\begin{equation}\label{eq:mu}
\mu=(\mu^i)\colon \bigoplus^t_{i=1}\Tot G_i\to M
\end{equation}
induces an isomorphism in homology.  For each $1\leq i\leq t$, define $\mu^i$ by a family $\mu^i_j\colon (\Tot G_i)_{j}\to M_j$ of homomorphisms of vector spaces over $k$ in such a way that each $\mu^i$ is $A$-linear and furthermore commutes with the differentials of $\Tot G_i$ and $M$.

For the sake of clarity, we include the following diagram, which illustrates the action of $\mu^i$ on $\Tot G_i$.  Note that the complexes are expressed vertically and that the diagonal maps represent the nontrivial component of $\dell^{\Tot G_i}$.

\begin{equation}\label{eq:XtoTotdiag}
\begin{split}\xymatrixcolsep{0.025pc}\xymatrixrowsep{0.5pc}
\xymatrix{
\vdots & & \vdots & & & & & & & & \vdots\ar@{->}[dd]\\
\oplus & & \oplus\\
(\shift^{c_i+1}A)_{c_i+2}\ar@{->}[ddrr]^{h_i} & \oplus & (\shift^{r_i}A)_{c_i+2} & & & & & & & & M_{c_i+2}\ar@{->}[dd]^{\dell^M_{c_i+2}}\ar@{<-}[llllllll]_{\mu^i_{c_i+2}}\\
\oplus & & \oplus\\
(\shift^{c_i+1}A)_{c_i+1}\ar@{->}[ddrr]^{h_i} & \oplus & (\shift^{r_i}A)_{c_i+1} & & & & & & & & M_{c_i+1}\ar@{->}[dd]^{\dell^M_{c_i+1}}\ar@{<-}[llllllll]_{\mu^i_{c_i+1}}\\
\oplus & & \oplus\\
0 & \oplus & (\shift^{r_i}A)_{c_i} & & & & & & & & M_{c_i}\ar@{->}[dd]^{\dell^M_{c_i}}\ar@{<-}[llllllll]_{\mu^i_{c_i}}\\
\oplus & & \oplus\\
\vdots & & \vdots & & & & & & & & \vdots\ar@{->}[dd]^{\dell^M_{r_i-1}}\\
& & \oplus\\
& & (\shift^{r_i}A)_{r_i} & & & & & & & & M_{r_i}\ar@{->}[dd]^{\dell^M_{r_i}}\ar@{<-}[llllllll]_{\mu^i_{r_i}}\\
& & \oplus\\
& & 0 & & & & & & & & M_{r_i-1}\ar@{->}[dd]\ar@{<-}[llllllll]_{\mu^i_{r_i-1}}\\
& & \oplus\\
& & \vdots & & & & & & & & \vdots
}\end{split}
\end{equation}

\medskip

Although this diagram is restricted to indices given by $1\leq i\leq s$, the situation is straightforward for $i>s$; indeed, these cases represent the torsion-free part of $\HH(M)$.  Thus, for $s<i\leq t$, one has the following.
\[
\mu^i_j\colon \left\{\begin{array}{ll}0\to M_j & \qquad \text{if }j<r_i\\[0.1cm]
(\shift^{r_i}A)_j\to M_j & \qquad \text{if }j\geq r_i
\end{array}\right.
\]

Now consider the following isomorphism of graded $A$-modules.
\[
\varphi\colon \bigoplus_{i=1}^s\frac{\shift^{r_i}A}{h_i\shift^{c_i}A}\oplus\bigoplus_{i=s+1}^t\shift^{r_i}A\to\HH(M)
\]
Fixing $1\leq i\leq t$, let $0\ne z_i\in M_{r_i}$ be a cycle defined in such a way that $\varphi\left(\susp^{r_i}1\right)=\cls(z_i)\in\HH_{r_i}(M)$.  Since $\varphi$ is $A$-linear, $\varphi\left(x^\ell\susp^{r_i}1\right)=\cls(x^\ell z_i)$ for $\ell\geq 0$.  Notice that, for small enough values of $\ell$, these classes are nonzero.  To be precise, $\cls(x^\ell z_i)=0$ if and only if $1\leq i\leq s$ and $\ell\geq c_i-r_i$.  Therefore, for each $1\leq i\leq s$, there exists $m_i\in M_{c_i+1}$ such that $\dell^M_{c_i+1}(m_i)=x^{c_i-r_i}z_i$.

We now proceed to define, for each $i$ and $j$, a basis $U^i_j$ for $(\Tot G_i)_j$ over $k$.
\[
U^i_j\supseteq\left\{
\begin{array}{ll}
\{x^{j-r_i}\susp^{r_i}1,\,(-1)^{j-c_i-1}x^{j-c_i-1}\susp^{c_i+1}1\}\quad & \mbox{if $1\leq i\leq s$ and $j> c_i$}\\[0.02in]
\{x^{j-r_i}\susp^{r_i}1\} & \mbox{otherwise}
\end{array}
\right.
\]

Finally, we define $\mu^i_j\colon (\Tot G_i)_j\to M_j$ in the following way.
\[
\mu^i_j(u)=\left\{\begin{array}{ll}
x^{j-r_i}z_i&\qquad \text{if }u=x^{j-r_i}\susp^{r_i}1\\[0.1cm]
x^{j-c_i-1}m_i&\qquad \text{if }u=(-1)^{j-c_i-1}x^{j-c_i-1}\susp^{c_i+1}1\\[0.1cm]
0&\qquad\hbox{otherwise}
\end{array}\right.
\]
By construction, $\mu^i=(\mu^i_j)$ is an $A$-linear degree zero chain map between $\Tot G_i$ and $M$.  Further, one can easily check that the Leibniz rule is satisfied, so that the map $\mu\colon \Tot F\to M$ given in \eqref{eq:mu} is a morphism of DG modules.  The above construction also guarantees that $\mu$ establishes a one-to-one correspondence between the generators of homology of $\Tot F=\bigoplus_{i=1}^m\Tot G_i$ and that of $M$.  The result follows.
\end{proof}

The following example illustrates the practical use of the construction used in the proof of Theorem \ref{thm:1var}.

\begin{example}
Let $M$ be the rank five semifree DG module over $A=k[x]$ with semibasis given by $\{e_1,e_2,e_3,e_4,e_5\}$ such that $|e_1|=0$, $|e_2|=2$, $|e_3|=4$, $|e_4|=8$, $|e_5|=9$, and where the differential of $M$ is defined by the following.
\begin{align*}
&\dell(e_1)=0\\
&\dell(e_2)=0\\
&\dell(e_3)=0\\
&\dell(e_4)=x^7e_1+x^5e_2\\
&\dell(e_5)=x^4e_3
\end{align*}
The homology of $M$ can be decomposed
\begin{align}\label{eq:decomhom}
\HH(M)&=\frac{Ae_1\oplus Ae_2\oplus Ae_3}{A(x^7e_1+x^5e_2)\oplus Ax^4e_3}\notag\\[0.1in]
&\cong\frac{A(x^2e_1+e_2)}{A(x^7e_1+x^5e_2)}\oplus\frac{Ae_3}{Ax^4e_3}\oplus Ae_1
\end{align}
whence one obtains the following deleted minimal free resolution $F$ of $\HH(M)$.
\[\xymatrixcolsep{2pc}\xymatrixrowsep{0pc}\xymatrix{
&0\ar@{->}[r]&\shift^7A\ar@{->}[r]^{x^5}&\shift^2A\ar@{->}[r]&0\\
&&\oplus&\oplus\\
F\colon &0\ar@{->}[r]&\shift^8A\ar@{->}[r]^{x^4}&\shift^4A\ar@{->}[r]&0\\
&&\oplus&\oplus\\
&&0\ar@{->}[r]&A\ar@{->}[r]&0}
\]

We shall now utilize the proof of Theorem \ref{thm:1var} to show that $\Tot F\simeq M$.  From $F$ we obtain the following subcomplexes.
\begin{align*}
G_1\colon \quad&0\to\shift^7A\xra{x^5}\shift^2A\to 0\\
G_2\colon \quad&0\to\shift^8A\xra{x^4}\shift^4A\to 0\\[0.05in]
G_3\colon \quad&0\to A\to 0
\end{align*}
Referring to the decomposition of homology in \eqref{eq:decomhom}, one has that the cycles generating $\HH(M)$ over $A$ are $z_1=x^2e_1+e_2$, $z_2=e_3$, and $z_3=e_1$.  Then, for each $j\in\NN$, a basis $U^1_j$ of $(\Tot G_3)_j$ over $k$ must be chosen to contain $\{x^j\susp^{0}1\}$.  Furthermore, the respective bases of $(\Tot G_1)_j$ and $(\Tot G_2)_j$ over $k$ are given as follows.
\[
\begin{array}{l}
U^1_j\supseteq\left\{
\begin{array}{ll}
\{x^{j-2}\susp^2 1\}\qquad & \text{if }j<8\\[0.02in]
\{x^{j-2}\susp^2 1,(-1)^{j-8}x^{j-8}\susp^8 1\} & \text{if }j\geq 8
\end{array}
\right.\\[0.15in]
U^2_j\supseteq\left\{
\begin{array}{ll}
\{x^{j-4}\susp^4 1\}\qquad & \text{if }j<9\\[0.02in]
\{x^{j-4}\susp^4 1,(-1)^{j-9}x^{j-9}\susp^9 1\} & \text{if }j\geq 9
\end{array}
\right.
\end{array}
\]

Using these bases, one can now define chain maps $\mu^i=(\mu^i_j)\colon (\Tot G_i)_{j}\to M_j$.
\begin{align*}
\mu^1_j(u)&=\left\{\begin{array}{p{1in}l}
$x^je_1+x^{j-2}e_2$&\text{if }u=x^{j-2}\susp^{2}1\\[0.1cm]
$x^{j-8}e_4$&\text{if }u=(-1)^{j-8}x^{j-8}\susp^{8}1\\[0.1cm]
$0$&\hbox{otherwise}
\end{array}\right.\\[0.15cm]
\mu^2_j(u)&=\left\{\begin{array}{p{1in}l}
$x^{j-4}e_3$&\text{if }u=x^{j-4}\susp^{4}1\\[0.1cm]
$x^{j-9}e_5$&\text{if }u=(-1)^{j-9}x^{j-9}\susp^{9}1\\[0.1cm]
$0$&\hbox{otherwise}
\end{array}\right.\\[0.15cm]
\mu^3_j(u)&=\left\{\begin{array}{p{1in}l}
$x^je_1$&\text{if }u=x^{j}\susp^0 1\\[0.1cm]
$0$&\hbox{otherwise}
\end{array}\right.
\end{align*}
Finally, $\mu\colon \Tot F\to M$ is given by $\mu(x)=\left(\mu^1(x),\mu^2(x),\mu^3(x)\right)$ for each $x\in\Tot F$.
\end{example}

\section*{Acknowledgments}

The author would like to thank Dave Jorgensen and Sean Sather-Wagstaff for both piquing and cultivating her interest in DG algebra.  Thanks also to the referee for many helpful suggestions which greatly improved the overall quality of this manuscript.

\bibliographystyle{amsxport}
\bibliography{mybib}

\end{document}

%% file: mymacros.tex
\newcommand{\DG}{\operatorname{\mathsf{DG}}}
\newcommand{\Gr}{\operatorname{\mathsf{Gr}}}

\newcommand{\Ch}{\operatorname{\mathbf{Ch}}\!}

\newcommand{\Der}{\operatorname{\mathbf{D}}\!}
\newcommand{\Tot}{\operatorname{Tot}}

\newcommand{\HH}{{\operatorname H}}
\newcommand{\dell}{\partial}

\newcommand{\rank}{\operatorname{rank}}

\newcommand{\ZZ}{\mathbb{Z}}
\newcommand{\NN}{\mathbb{N}}

\newcommand{\xra}{\xrightarrow}

\newcommand{\var}{{\hskip1pt\vert\hskip1pt}}

\newcommand{\cls}{\operatorname{cls}}

\newcommand{\shift}{\Sigma}
\newcommand{\susp}{\sigma}

\theoremstyle{plain}
\newtheorem{theorem}{Theorem}[section]

\newtheorem{proposition}[theorem]{Proposition}
\newtheorem{lemma}[theorem]{Lemma}
\newtheorem{corollary}[theorem]{Corollary}

\theoremstyle{definition}
\newtheorem{definition}[theorem]{Definition}

\newtheorem{example}[theorem]{Example}

\theoremstyle{remark}
\newtheorem{remark}[theorem]{Remark}

\newtheoremstyle{myplain}
     {1em plus 0.15em minus 0.05em}
     {1em plus 0.15em minus 0.05em}
     {\itshape}
     {}
     {\bf}
     {.}
     {0.5em}
     {}

\newtheoremstyle{mydef}
     {1em plus 0.15em minus 0.05em}
     {1em plus 0.15em minus 0.05em}
     {}
     {}
     {\bf}
     {.}
     {0.5em}
     {}

\newtheoremstyle{myrmk}
     {1em plus 0.15em minus 0.05em}
     {1em plus 0.15em minus 0.05em}
     {}
     {}
     {\itshape}
     {.}
     {0.5em}
     {}

\theoremstyle{myplain}

\theoremstyle{myrmk}

\theoremstyle{mydef}